\documentclass{LMCS}

\def\dOi{10(4:16)2014}
\lmcsheading%
{\dOi}
{1--23}
{}
{}
{Nov.~\phantom06, 2013}
{Dec.~24, 2014}
{}

\ACMCCS{[{\bf Theory of computation}]:  Logic---Constructive
  mathematics; [{\bf Mathematics of computatiing}]: Continuous
  mathematics---Calculus} 

\subjclass[2012]{Theory of computation~Constructive mathematics; Mathematics of computing~Calculus}

\usepackage[english]{babel}
\usepackage[utf8]{inputenc}

\usepackage[T1]{fontenc}
\usepackage{lmodern}
\usepackage{hyperref}

\usepackage{verbatim}
\usepackage{url}
\usepackage{xcolor}
\usepackage{microtype}
\usepackage{varwidth} 
\usepackage{mathtools}
\usepackage{tikz}
\usetikzlibrary{positioning, shapes, calc}
\makeatletter
\global\let\tikz@ensure@dollar@catcode=\relax
\makeatother

\usepackage[LogicalSystemsSans]{ak_logic}

\newcommand*{\Theorem}{Theorem}
\newcommand*{\Proposition}{Proposition}
\newcommand*{\Lemma}{Lemma}
\newcommand*{\Corollary}{Corollary}
\newcommand*{\Definition}{Definition}
\newcommand*{\Question}{Question}
\newcommand*{\Remark}{Remark}
\newcommand*{\Notation}{Notation}
\newcommand*{\Figure}{Figure}

\theoremstyle{plain}
\newtheorem{theorem}{\Theorem}
\newtheorem{proposition}[theorem]{\Proposition}
\newtheorem{corollary}[theorem]{\Corollary}
\newtheorem{lemma}[theorem]{\Lemma}
\theoremstyle{definition}
\newtheorem{definition}[theorem]{\Definition}

\theoremstyle{remark}
\newtheorem{remark}[theorem]{\Remark}

\usepackage{prettyref}
\newrefformat{thm}{\Theorem~\ref{#1}}
\newrefformat{pro}{\Proposition~\ref{#1}}
\newrefformat{cor}{\Corollary~\ref{#1}}
\newrefformat{lem}{\Lemma~\ref{#1}}
\newrefformat{rem}{\Remark~\ref{#1}}
\newrefformat{def}{\Definition~\ref{#1}}
\newrefformat{fig}{\Figure~\ref{#1}}

\title[Bounded variation and the strength of Helly's selection theorem]{Bounded variation and the strength of \\ Helly's selection theorem}
\author[A.~P.~Kreuzer]{Alexander P.\ Kreuzer}
\address{Department of Mathematics \\
Faculty of Science \\
National University of Singapore \\
Block S17, 10 Lower Kent Ridge Road \\
Singapore 119076 
}
\email{matkaps@nus.edu.sg}
\urladdr{\url{http://www.math.nus.edu.sg/~matkaps/}}
\thanks{The author was partly supported by the RECRE project, and the Ministry of Education of Singapore through grant R146-000-184-112 (MOE2013-T2-1-062).}


\begin{document}

\begin{abstract}
  We analyze the strength of Helly's selection theorem (\lp{HST}), which is the most important compactness theorem on the space of functions of bounded variation ($BV$). For this we utilize a new representation of this space intermediate between $L_1$ and the Sobolev space $W^{1,1}$, compatible with the---so called---weak$^*$ topology on $BV$. We obtain that \lp{HST} is instance-wise equivalent to the Bolzano-Weierstra{\ss} principle over \ls{RCA_0}. With this \ls{HST} is equivalent to \ls{ACA_0} over \ls{RCA_0}. A similar classification is obtained in the Weihrauch lattice.
\end{abstract}

\maketitle

In this paper we investigate the space of functions of bounded variation ($BV$) and Helly's selection theorem (\lp{HST}) from the viewpoint of reverse mathematics and computable analysis. Helly's selection theorem is the most important compactness principle on $BV$. It is used in analysis and optimization, see for instance \cite{AFP00,BP12}.

This continues our work in \cite{aK11} and \cite{aK14a} where (instances of) the Bolzano-Weierstra{\ss} principle and the  Arzelà-Ascoli theorem were analyzed. There we showed, among others, that an instance of the Arzelà-Ascoli theorem is equivalent to a suitable single instance of the Bolzano-Weierstra{\ss} principle (for the unit interval $[0,1]$), which, in turn, is equivalent to an instance of \lp{WKL} for $\Sigma^0_1$-trees.
Here, we will show that an instance of Helly's selection theorem is equivalent to a single instance of the Bolzano-Weierstra{\ss} principle (and with this to an instance of the other principles mentioned above).
It is a priori not clear that this is possible since the proof of \lp{HST} uses seemingly iterated application of the Arzelà-Ascoli theorem and since there are compactness principles, which are instance-wise strictly stronger than Bolzano-Weierstra{\ss} for $[0,1]$. (For instance the Bolzano-Weierstra{\ss} principle for weak compactness on $\ell_2$ has this property, see \cite{aK12c}.)
A fortori this shows that \lp{HST} is equivalent to \ls{ACA_0} over \ls{RCA_0}.

We represent $BV$ as a weak derivative space in the style of Sobolev spaces. 
Our representation differs from all previous treatments in computable analysis or constructive mathematics known to the author.
Previously functions of bounded variation were regarded as actual functions, whereas we only regard them as $L_1$\nobreakdash-functions. 
With this, they can be characterized by the integral of absolute value of their weak derivative.
This has the advantage that it is closer to modern applications.
Moreover, this allows one to easily define functions of bounded variation not only on the real line but also on $\Real^n$, which is not possible with the classical definition of bounded variation.
We therefore believe that our representation has also other applications in computable analysis.

This paper is organized as follows. In \prettyref{sec:int} we define the space $BV$, in \prettyref{sec:comp} we compare $BV$ to other spaces and to other possible representations of functions of bounded variation, and in \prettyref{sec:hst} we analyze Helly's selection theorem.

\section{The space of functions of bounded variation}\label{sec:int}

A \emph{countable vector space} $A$ over a countable field $K$  consists of a set $\lvert A \rvert\subseteq \Nat$ and mappings $+ \colon \lvert A \rvert \times \lvert A \rvert \longto \lvert A \rvert$,  $\cdot \colon K \times \lvert A \rvert \longto \lvert A \rvert$, and a distinguished element $0\in \lvert A \rvert$, such that $A,+,\cdot,0$ satisfies the usual vector space axioms.

A (code for a) \emph{separable Banach space} $B$ consists of a countable vector space $A$ over $\Rat$ together with a function $\lVert \cdot \rVert \colon A \longto \Real$ satisfying $\lVert q \cdot a \rVert = \lvert q \rvert \cdot \lVert a \rVert$ and $\lVert a + b \rVert \le \lVert a \rVert + \lVert b \rVert$ for all $q\in \Rat$, $a,b\in A$. A point in $B$ is defined to be a sequence of elements $(a_k)_k$ in $A$ such that $\lVert a_k - a_{k+1}\rVert \le 2^{-k}$.
Addition and multiplication on $B$ are defined to be the continuous extensions of $+$, $\cdot$ from $A$ to $B$.

The space $L_1 := L_1([0,1])$ will be represented by the $\Rat$-vector space of rational polynomials $\Rat[x]$ together with the norm ${\lVert p \rVert}_{1} := \int_0^1 \lvert p(x) \rvert\, dx$. Since the rational polynomials are dense in the usual space $L_1$, this defines (a space isomorphic to) the usually used space (provably in suitable higher-order system where the textbook definition of $L_1$ can be formalized). See Example~II.10.4, Exercise~IV.2.15 and Chapter~X.1 in \cite{sS09}.

\subsection {Bounded variation}

The \emph{variation} of a function $f\colon [0,1]\longto \Real$ is defined to be 
\begin{equation}\label{eq:var}
V(f) := \sup_{0\le t_1 < \dots < t_n \le 1} \sum_{i=1}^{n-1} \abs{f(t_i) - f(t_{i+1})}
.\end{equation}
For an $L_1$-equivalence classes of functions $f\in L_1$ the variation is defined to be the infimum over all elements, i.e.,
\begin{equation}\label{eq:varl}
V_{L_1}(f) := \inf \left\{\, V(g) \sizeMid g\colon [0,1]\to \Real \text{ and $g=f$ almost everywhere} \,\right\}
.\end{equation}

The subspace of all $L_1$\nobreakdash-functions of bounded variation form a subspace of $L_1$ with the following norm
\begin{equation}\notag
{\lVert f \rVert}_{BV} := {\lVert f \rVert}_{1} + V_{L_1}(f)
.\end{equation}
However, it is not possible to code this space as a separable Banach space, as we did for $L_1$, since the variation $V$ is difficult to compute (see \prettyref{pro:bccduleq1eq} below) and since this space is not separable in this norm. (To see this take for instance the characteristic functions $\chi_{[0,u]}(x)$ of the intervals $[0,u]$. It is clear that these functions belong to $BV$. For $u,w\in [0,1]$ with $u \neq w$ the function $\chi_{[0,u]} - \chi_{[0,w]}$ contains a bump of height $1$, therefore ${\lVert \chi_{[0,u]} - \chi_{[0,w]}\rVert}_{BV} \ge 2$. Thus, these functions form a set of the size of the continuum which cannot be approximated by countably many functions.)

We will define the space $BV$ to be a subspace of $L_1$. 
\begin{definition}[$BV$, \ls{RCA_0}]\label{def:bv}
  The space $BV:=BV([0,1])$ is defined like the space $L_1([0,1])$ with the following exception.
  A point in $BV$ is a sequence $(p_k)_k\subseteq \Rat[x]$ together with a rational number $v\in \Rat$, such that
  \begin{itemize}
  \item ${\lVert p_k - p_{k+1} \rVert}_{1} \le 2^{-k}$, and
  \item $\int_0^1 \lvert p_k'(x) \rvert \, dx \le v$.
  \end{itemize}

  The vector space operations are defined pointwise for $p_k$ and $v$. (For scalar multiplication one chooses a suitable rational upper bound for the new $v$.) 
  
  The parameter $v$ will be called the \emph{bound on the variation of $f$}.
\end{definition}

This definition is justified by Propositions \ref{pro:jus1} and \ref{pro:jus2} below.
For later use we will collect the following lemma.
\begin{lemma}[\ls{RCA_0}]\label{lem:l1bv}
  Let $(f_n)_n\subseteq BV$ be a sequence converging in $L_1$ at a fixed rate to a function $f\in L_1$, i.e., ${\lVert f_n-f \rVert}_1 \le 2^{-n}$.
  If the bounds of variations $v_n$ for $f_n$ are uniformly bounded by a $v$, then $f\in BV$.
\end{lemma}
\begin{proof}
  Let $(p_{n,k})_k$ be the rational polynomials coding $f_n$. One has ${\lVert p_{k+1,k+1} - f\rVert}_1 \le {\lVert p_{k+1,k+1} - f_{k+1}\rVert}_1 + {\lVert f_{k+1} - f\rVert}_1 \le 2^{-k}$. Thus, $(p_{k+1,k+1})_k, v$ is a code for $f$ in the sense of \prettyref{def:bv}.
\end{proof}

For working with functions of $BV$ it will be handy to use mollifiers as defined below, since one can use them to smoothly approximate characteristic functions without increasing the variation.

\begin{definition}[Mollifier, \ls{RCA_0}]
  Let 
  \[\eta(x) :=
  \begin{cases}
    c \cdot \exp\left(\frac{1}{x^2-1}\right) & \text{if }\lvert x \rvert < 1\text{,} \\
    0 & \text{otherwise,}
  \end{cases}
  \quad
  \text{where }c:= \left( \int_{-1}^1 \exp\left(\frac{1}{x^2-1}\right) \, dx\right)^{-1}.
  \]
  The function $\eta$ is called a \emph{mollifier}.
  It is easy to see that $\eta$ is infinitely often differentiable provably in $\ls{RCA_0}$. By definition $\int_{-1}^1 \eta \, dx = 1$.

  Define  $\eta_\epsilon(x) := \frac{1}{\epsilon} \eta\left(\frac{x}{\epsilon}\right)$.
  We have that the support of $\eta_\epsilon$ is contained in $B(0,\epsilon)=\{ x\in \Real \mid \abs{x} < \epsilon \}$ and that $\int_{-1}^1 \eta_\epsilon \, dx = 1$.
\end{definition}
The integral of this mollifier can be used to smoothly approximate characteristic functions of intervals. For instance
\begin{equation}\label{eq:chiappr}
  x\longmapsto \int_{-1}^x \eta_\epsilon\left(y-\tfrac{1}{4}\right) - \eta_\epsilon\left(y-\tfrac{3}{4}\right)\, dy
\end{equation}
approximates $\chi_{\left[\frac{1}{4},\frac{3}{4}\right]}$ in $L_1$, see \prettyref{fig:chiappr}. Since the approximating function does not oscillate, the variation of it is not bigger that the variation of the approximated function.
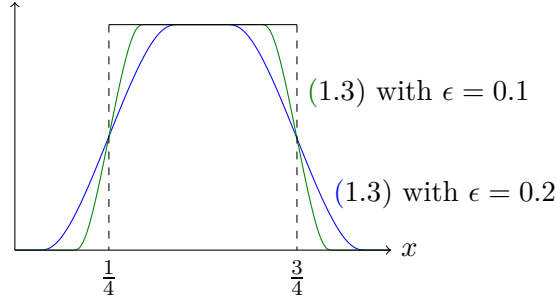
\begin{figure}
  \centering
  \begin{tikzpicture}[xscale=5, yscale=3]
    \draw[color=blue] plot[id=mol1,smooth] file{bv.mol1.table};
    \draw[color=blue] (0.82,0.25) node[right] {\eqref{eq:chiappr} with $\epsilon=0.2$};
    \draw[color=green!50!black] plot[id=mol2,smooth] file{bv.mol2.table};
    \draw[color=green!50!black] (0.75,0.7) node[right] {\eqref{eq:chiappr} with $\epsilon=0.1$};

    \draw (0.25,1)--(0.75,1);
    \draw[dashed] (0.25,1)--(0.25,0) node[below] {$\tfrac{1}{4}$};
    \draw[dashed] (0.75,1)--(0.75,0) node[below] {$\tfrac{3}{4}$};

    \draw[->] (0,0)--(1.0,0) node[right] {$x$};
    \draw[->] (0,0)--(0,1.1);
  \end{tikzpicture}
  \caption{Approximation of $\chi_{[\frac{1}{4},\frac{3}{4}]}$.}\label{fig:chiappr}
\end{figure}

The integral of such a mollifier $x\longmapsto \int_0^x \eta_\epsilon(y-z)\, dy$ is contained in $BV$. To see this let $(q_k)_k \subseteq \Rat[x]$ be a sequence approximating $\eta_\epsilon(x-z)$ in $L_1$, i.e.
\[
{\lVert q_k - \eta_\epsilon(x-z) \rVert}_{1} \le 2^{-k}
.\]
Since ${\lVert \eta_\epsilon(x-z) \rVert}_{1} \le 1$ we have that $\lVert q_k \rVert \le 2$.
Integrating $q_k$ we obtain a sequence of again rational polynomials $p_k(x) = \int_0^x q_k(y) \, dy$. By definition ${\left\lVert p_k-  \int_0^x \eta_\epsilon(y-z)\, dy\right\rVert}_{1} \le 2^{-k}$. Thus $(p_k)_k, v=2$ is a code for the integral of the mollifier.

\begin{proposition}[{\ls{WWKL_0}}]\label{pro:just1con}
   Let $f\colon [0,1]\longto \Real$ be a continuous function. If the variation of $f$ is bounded, that means that there exists a $v\in \Rat$ such that all sums in \eqref{eq:var} are bounded by $v$, then (the $L_1$-equivalence class of) $f$ belongs to $BV$.
\end{proposition}

For the proof of this proposition we will need the following notation and theorem from \cite{SY12}.
A partition of $[0,1]$ is a finite set $\Delta=\big\{\, 0=x_0 \le \xi_1 \le x_1 \le \dots \le \xi_n \le x_n = 1 \,\big\}$. The mesh of $\Delta$ is $\lvert \Delta\rvert:=\max\{x_k -x_{k-1} \mid 1\le k\le n\}$.
The Riemann sum for $\Delta$ is $S_\Delta(f):= \sum_{k=1}^n f(\xi_k)(x_k-x_{k-1})$. 
The limit $\lim_{\lvert \Delta \rvert\to 0} S_\Delta(f)=\int_0^1 f(x)\,dx$ is the Riemann integral.

\begin{definition}
  A function $f$ is effectively integrable if there exists a $h\colon \Nat \longto \Nat$ such that for any partitions $\Delta_1,\Delta_2$ and  $n\in\Nat$,
  \[
  \lvert \Delta_1 \rvert < 2^{-h(n)} \AND \lvert \Delta_2 \rvert < 2^{-h(n)} \IMPL \lvert S_{\Delta_1(f)} - S_{\Delta_2(f)} \rvert < 2^{-n+1}
  .\]
  The function $h$ is called modulus of integrability for $f$.
\end{definition}
\begin{theorem}[\ls{RCA_0}, \cite{SY12}]\label{thm:uniint}
  The following are equivalent:
  \begin{enumerate}
  \item \lp{WWKL_0},
  \item Every bounded, continuous function on $[0,1]$ is effectively integrable.
  \end{enumerate}
\end{theorem}

\begin{proof}[Proof of \prettyref{pro:just1con}]
  Since the variation of $f$ is bounded, $f$ is bounded. Therefore by \prettyref{thm:uniint} the function $f$ is effectively integrable. In particular, there exists a modulus of integrability $h$.

  Let $f_n$ be the following sequence of step functions approximating $f$.
  \[
  f_n(x) := \sum\nolimits_k \chi_{[k \cdot 2^{-h(n)},(k+1) \cdot 2^{-h(n)})}  \cdot f(k \cdot 2^{-h(n)})  \quad \text{where $k$ is such that } x\in \left[ \tfrac{k}{2^{h(n)}}, \tfrac{k+1}{2^{h(n)}}\right)
  \]
  Since $f_n$ is a finite sum of characteristic functions of intervals, it belongs to $BV$.  The variation of $f_n$ is obviously bounded by $v$.
  By definition $\lVert f_n - f_{n+1}\rVert_1 < 2^{-n+1}$, thus $(f_n)_n$ converges in $L_1$-norm to an $f\in L_1$. 
  By \prettyref{lem:l1bv}, $f\in BV$.
\end{proof}

In the following we will use right continuous functions. Such a function $f\colon [0,1]\longto \Real$ will be coded by a sequence of real numbers $(x_q)_{q\in\Rat}$ index by rational numbers such that the limit from the right
\[
\lim_{q\searrow x,q\in \Rat} x_q =: f(x)
\]
exists. This definition makes sense in \ls{ACA_0}.

\begin{proposition}[\ls{ACA_0}]\label{pro:jus1}
  Let $f\colon [0,1]\longto \Real$ be a right continuous function.
  If the variation of $f$ is bounded as in \prettyref{pro:just1con} then (the $L_1$-equivalence class of) $f$ belongs to $BV$.
\end{proposition}
\begin{proof}
  We approximate $f$ using the functions $f_n$ given by
  \[
  f_n(x) := \sum\nolimits_k \chi_{[k \cdot 2^{-n},(k+1) \cdot 2^{-n})}  \cdot f(k \cdot 2^{-n})  \quad \text{where $k$ is such that } x\in \left[ \tfrac{k}{2^n}, \tfrac{k+1}{2^n}\right)
  \]
  Like in the proof of \prettyref{pro:just1con} the variation of $f_n$ is bounded by the variation $v$ of $f$. The values of $f_n(x)$ are included in $[f(0)-v,f(0)+v]$.
  The functions $f_n(x)$ converge to $f$ on all points of continuity of $f$. We claim that the points of discontinuity of $f$ have measure $0$.
  Indeed, consider the measurable set (in the sense of \cite[Defintion~X.1.12]{sS09})
  \[
  A:= \bigcup_{n\in\Nat} \underbrace{\bigcap_{k\in \Nat} \left\{\, x \sizeMid  \max \left(\abs{f(x-2^{-k})-f(x)}, \abs{f(x+2^{-k})-f(x)}\right) > 2^{-n} \,\right\}}_{=:A_n}
  .\]
  This formula describes the points of discontinuity of $f$.
  Consider the set $A_n$ from above. If for any $n$ the set $A_n$ would have positive measure then there exists $2^n \cdot v$ many points in $A_n$ which would contradict the boundedness of the variation. Thus each $A_n$ has measure $0$ and with this $A$.
  Therefore, we can apply the dominated convergence theorem (see \cite[Theorem~4.3]{ADR12}) and obtain that $(f_n)_n$ converges in $L_1$ to (the $L_1$-equivalence class of) $f$ and by \prettyref{lem:l1bv} then $f\in BV$.
\end{proof}

\begin{lemma}[\lp{RCA_0}]
  For a continuous function $f\colon [0,1]\longto \Real$, such that $\abs{f'(x)}$ is effectively integrable, the variation $V(f)$ is bounded by $\int_0^1 \abs{f'(x)} \, dx$.

\end{lemma}
\proof
  For two points $t_1,t_2\in [0,1]$ we can estimate
  \[
  \abs{f(t_1)-f(t_2)} = \abs{\int_{t_1}^{t_2} f'(x)\, dx} \le \int_{t_1}^{t_2} \abs{f'(x)}\, dx
  .\]
  Therefore,
  \begin{align*}
    V(f) & = \sup_{0\le t_1 < \dots < t_n \le 1} \sum_{i=1}^{n-1} \abs{ f(t_i) - f(t_{i+1})} \\
    & \le \sup_{0\le t_1 < \dots < t_n \le 1} \sum_{i=1}^{n-1} \int_{t_i}^{t_{i+1}} \abs{f'(x)}\, dx 
    \le \int_0^1 \abs{f'(x)} \, dx\rlap{\hbox to 84 pt{\hfill\qEd}}
  .\end{align*}

\begin{proposition}[\ls{ACA_0}]\label{pro:jus2}
  For each $f\in BV$ there exists a right-continuous function $g$ which is almost everywhere equal to $f$ and with $V(g)<\infty$, or in other words the infimum in \eqref{eq:varl} is bounded.
\end{proposition}
\begin{proof}
  Let $(p_k)_k$, $v$ be a code for $f$.  By the previous lemma $V(p_k) \le v$.

  By \cite[Remark~X.1.11]{sS09} the polynomials $(p_k)_k$ converge to a function $g$ almost everywhere. To be precise there exists an ascending sequence of closed sets ${(C^f_n)}_n$ with measure $1-2^{-n}$ such that $(p_k(x))_k$ converges uniformly on  $C^f_n$ for each $n$. Let $M:=\bigcup_n C^f_n$.
  It is clear that $(p_k)_k$ converges to $g$ also in $L_1$-norm.

  The variation of $g$ with $t_i$ in \eqref{eq:var} restricted to be in $M$ is, as the pointwise limit of $p_k$, also bounded by $v$.

  To obtain the proposition the only thing left to show is how to extend $g$ to a proper function on the full unit interval.
  We claim that there exists a subsequence of ${(p_{k_j})}_j$ such that ${(p_{k_j}(x))}_j$ converges for all $x\in \Rat \cap [0,1]$. 
  To obtain this subsequence note that $\abs{p_k(x)} \le {\lVert f \rVert}_1 + v=:v'$. Let $q_i$ be an enumeration of $\Rat \cap [0,1]$ and consider for each $k$ the point $(p_k(q_i))_i \in {[-v',v']}^\Nat$. 
  Now ${[-v',v']}^\Nat$ is compact and $((p_k(q_i))_i)_k$ contains, by the Bolzano-Weierstra{\ss} principle, a convergent subsequence, which also satisfies the claim. See Lemma~III.2.5 and Theorem~III.2.7 of \cite{sS09}.

  Thus, we may assume that $\Rat \cap [0,1]\subseteq M$ by passing to a subsequence of $(p_k)$.
  Then let $g_+$ be the right continuous extension of $g$, i.e.
  \[
  g_+(x) := 
  \begin{cases}
    g(x) & \text{if } x\in M , \\
    \lim_{y\searrow x, \, y\in \Rat} g(y) & \text{otherwise.}
  \end{cases}
  \]

  The limit in the second case exists by the boundedness of the variation of $g$. Suppose that it does not exist then there would be an $\epsilon$ and an infinite sequence in $M$ oscillating at least $\epsilon$ at each step and, with this, the variation of $g$ would be infinite.
  
  The almost everywhere converging subsequence of $(p_k)_k$ follows by Remark~X.1.11~\cite{sS09} from \lp{WWKL}. The set $M$ is arithmetic and thus exists provably in \ls{ACA_0}. Also the extension $g_+$ of $g$ can be build in using a routine application of the Bolzano-Weierstra{\ss} principle again provable in \ls{ACA_0}.
\end{proof}

\begin{corollary}[Jordan decomposition, \ls{ACA_0}]\label{cor:jordan}
  For each function $f\in BV$ coded by $(p_k)_k,v$ there exists a
  measurable set $C$ such that $f$ restricted to $C$ is
  non-decreasing, that is, $\liminf p'_k(x) \ge 0$ for almost all
  $x\in C$, and $f$ restricted to the complement of $C$ is
  non-increasing, that is, $\limsup p'_k(x) \le 0$.
\end{corollary}
\proof
  Let $g$ be the right-continuous function as in \prettyref{pro:jus2} and let $C$ be the following measurable set
  \begin{align*}
    C &:= \bigcap_{m\in\Nat} \bigcup_{n\in\Nat} \bigcap_{k>n} \{ x \mid g(x) < g(x+2^{-k}) + 2^{-m} \}
    .
    \intertext{Since $g$ has bounded variation the complement of $C$ is almost everywhere equal to }
    [0,1]\setminus C &= \bigcap_{m\in\Nat} \bigcup_{n\in\Nat} \bigcap_{k>n} \{ x \mid g(x) > g(x+2^{-k}) - 2^{-m}  \}
    .
  \end{align*}
  The result follows.\qed
Independently, the Jordan decomposition was investigated by Nies, Yokoama et al. in \cite{LogicBlog2013}.

\section{Comparison to other spaces}\label{sec:comp}

\subsection{Sobolev space $W^{1,1}$}

Our motivation for representing the space $BV$ in the way we did in \prettyref{def:bv} is that in this way $BV$ lies between $L_1$ and the Sobolev space $W^{1,1}$. We believe that this is the right way to represent this space since $BV$ is in practice almost always used as an intermediate space between $L_1$ and $W^{1,1}$. 

Recall that the Sobolev space $W^{1,1}:= W^{1,1}([0,1])$ is the coded separable Banach space over the rational polynomials $\Rat[x]$ together with the following norm
\[
{\lVert p \rVert}_{W^{1,1}} := {\lVert p \rVert}_{1} + {\lVert p' \rVert}_{1}
.\]
From this definition it is obvious that $W^{1,1}$ is a subspace of $BV$.

\begin{proposition}[\ls{RCA_0}]
  $
  W^{1,1} \subseteq  BV \subseteq L^1
  $
  and all of these  inclusions are strict.
\end{proposition}
\begin{proof}
  The inclusions are clear. We show only the strictness.
  The function 
  \[
  f(x) :=
  \begin{cases}
    x \cdot \sin(1/x \cdot 2\pi) & \text{if $x>0$,}  \\
    0 & \text{otherwise,}
  \end{cases}
  \]
  is continuous on $[0,1]$ and therefore contained in $L^1$. However, it has unbounded variation and therefore $f\notin BV$.
  A characteristic function of a nontrivial interval, say $\chi_{[\frac{1}{2},1]}$, is contained in $BV$. It is not contained in $W^{1,1}$, because the derivative of $\chi_{[\frac{1}{2},1]}$ would be almost everywhere $0$ and infinite at $\frac{1}{2}$, which is impossible.
\end{proof}

\subsection{$BV$ as dual space}

It is well known that the space $BV$ is isomorphic to the dual space of $C([0,1])$, i.e.\ the space of uniformly continuous and  linear functionals defined on the continuous functions on $[0,1]$ with ${\lVert \cdot \rVert}_\infty$-norm.
Before we can show this we will need some more properties of mollifiers.
\begin{definition}[Mollification of a function, \ls{RCA_0}]
  Let $f\colon [0,1]\longto \Real$ be a continuous, effectively integrable function.
  We extend $f$ to $[-1,2]$ by setting $f(x) = f(1-x)$ for $x > 1$ and $f(x) = f(-x)$ for $x < 0$. 
  We define the \emph{mollification of $f$} to be
  \begin{equation}\label{eq:defmollification}
  f^\epsilon(x) := (f \ast \eta_\epsilon)(x) := \int_{x-\epsilon}^{x+\epsilon} \eta_\epsilon(x-y) f(y) \, dy = \int_{-\epsilon}^{\epsilon} \eta_\epsilon(y)f(x-y) \, dy
  \end{equation}
  for $x\in [0,1]$ and $0< \epsilon \le 1$.

  For a function $f\in L_1$ the mollification is defined in the same way. (The extension of $f$ can be defined pointwise for each $(p_k)_k$ coding $f$.)
\end{definition}
\begin{proposition}[\ls{RCA_0}]
  Let $f$ be as above.
  \begin{enumerate}[label=(\roman*)]
  \item\label{enum:molprop:1} $f^\epsilon$ is infinitely often differentiable.
  \item\label{enum:molprop:2} If $f$ is uniformly continuous, then $f^\epsilon \xrightarrow{\epsilon \to 0} f$ uniformly. If $f$ has additionally a modulus of uniform continuity then there exists a modulus of convergence for $f^\epsilon \xrightarrow{\epsilon\to 0} f$.
  \end{enumerate}
\end{proposition}
\begin{proof}
  \ref{enum:molprop:1}: We show only that $f^\epsilon$ differentiable.
  \begin{equation*}
    \frac{f^\epsilon(x+h) - f^\epsilon(x)}{h} = \frac{1}{\epsilon} \int_0^1 \frac{1}{h} \left( \eta\left(\frac{x+h-y}{\epsilon}\right) - \eta\left(\frac{x-y}{\epsilon}\right)\right) f(y) \, dy
  \end{equation*}
  Now for $h\to 0$ we have that $\frac{1}{h}\left(\eta\big(\frac{x+h-y}{\epsilon}\big) - \eta\big(\frac{x-y}{\epsilon}\big)\right)$ converges uniformly in $y$ to $\frac{d}{dx} \eta\left(\frac{x-y}{\epsilon}\right)$. Therefore one can exchange integration and taking the limit of $h$ and obtains that 
  \begin{equation}\label{eq:moldif2}
    \frac{d}{dx} f^\epsilon(x) = \frac{1}{\epsilon} \int_0^1 \frac{d}{dx}\big(\eta(x-y)\big) \cdot f(y)\, dy
  \end{equation}
  exists.

  \ref{enum:molprop:2}: 
  \begin{align*}
    \abs{f^\epsilon(x) - f(x)} &= \abs{\int_{x-\epsilon}^{x+\epsilon} \eta_\epsilon(x-y) (f(y)-f(x)) \, dy}  \\
      & \le \sup_{y\in [x-\epsilon,x+\epsilon]} \abs{f(y)-f(x)}\  \xrightarrow{\epsilon\to 0} 0 \qquad \text{by uniform continuity}.
  \end{align*}
  Thus from a modulus of uniform continuity one can define a uniform modulus of convergence of $f^\epsilon(x) \xrightarrow{\epsilon \to 0} f(x)$.
\end{proof}

For a code $(p_k)_k, v$ for an $f\in BV$ let $T$ be the following linear functional defined on all $h\in C([0,1])$.
\begin{align}
T(h) & := \lim_{k\to \infty} \int_0^1 h \cdot p_k' \, dx  \label{eq:cdinf}
\shortintertext{
  Note that $T$ will depend not only on the $L_1$-class of $f$ but also on the specific sequence of rational polynomials. See \prettyref{pro:intpart} below.
  We can estimate}
& T(h) \le {\lVert h \rVert}_\infty \cdot v .\notag
\end{align}
Thus $T$ is continuous and therefore in the dual $C^*([0,1])$. It is clear that this is provable in \ls{ACA_0}. (For a formal definition of bounded functionals and the dual space, see Definitions~II.10.5 and X.2.3 in \cite{sS09}.)

For the other direction let $T\colon C([0,1]) \longto \Real$ be a linear, continuous functional with $\lVert T\rVert \le v$ for some $v\in \Real$. 
We can continuously  extend $T$ to functions of the form $\chi_{(y,1]}$ (and linear combinations thereof) by approximating this function using the mollifier, cf.~\eqref{eq:chiappr}.
We claim that the function 
\[
m(y) := T(\chi_{(y,1]})
\]
has bounded variation. Indeed for $0\le t_1 < \dots < t_n \le 1$ we have
\begin{align*}
  \sum_{i=1}^{n-1} \abs{m(t_{i+1}) - m(t_i)} &= \sum_{i=1}^{n-1} e_i\, {\left(m(t_{i+1}) - m(t_i)\right)} \hphantom{\le v} \text{for suitable }e_i\in \{-1,1\} \\
  & = T\left(\sum_{i=1}^{n-1} e_i \, \chi_{(t_i,t_{i+1}]}\right)  \\
  & \le v \hphantom{= \sum_{i=1}^{n-1} e_i\, {\left(m(t_{i+1}) - m(t_i)\right)}} \text{since the sum is bounded by $1$.}
\end{align*}
It is clear that $m$ is right continuous. Thus, by \prettyref{pro:jus1} we have $m\in BV$. Now let $h\in C([0,1])$ be a uniformly continuous function. The function $h$ can be approximated in ${\lVert \cdot \rVert}_\infty$ by functions of the form 
\[
h_n(x) :=  h\left(\frac{i}{2^n}\right) \quad\text{if } x\in\left[\frac{i}{2^n},\frac{i+1}{2^n}\right)
.\]
(A modulus of convergence can be defined from a modulus of uniform continuity of $h$.)
Then 
\begin{align*}
  T(h) & = T\left(\lim_{n\to \infty} h_n\right) = \lim_{n\to\infty} T(h_n)  \\
  & = \lim_{n\to \infty} \sum_i \left[h\!\left(\frac{i}{2^n}\right) \cdot  \left(m\!\left(\frac{i+1}{2^n}\right)-m\!\left(\frac{i}{2^n}\right)\right)\right] 
  \shortintertext{(for a suitable choice of $(p_k)_k$ converging pointwise at all $q\in[0,1]\cap\Rat$, see proof of \prettyref{pro:jus2})}
  & = \lim_{n\to\infty} \lim_{k\to\infty}\sum_i \left[h\!\left(\frac{i}{2^n}\right) \cdot  \left(p_k\!\left(\frac{i+1}{2^n}\right)-p_k\!\left(\frac{i}{2^n}\right)\right)\right] 
  \shortintertext{(by uniform convergence in $n$)}
  & =  \lim_{k\to\infty} \lim_{n\to\infty}\sum_i \left[h\!\left(\frac{i}{2^n}\right) \cdot  \left(p_k\!\left(\frac{i+1}{2^n}\right)-p_k\!\left(\frac{i}{2^n}\right)\right)\right] \\
& = \lim_{k\to \infty} \int_0^1 h \cdot p_k' \, dx .
\end{align*}

These observations give rise to the following propositions.
\begin{proposition}[\ls{ACA_0}]\label{pro:bccduleq1}
  Each (code of an) $f\in BV$ induces a bounded linear functional $T\in C^*([0,1])$ given by~\eqref{eq:cdinf}.
\end{proposition}
\begin{proposition}[\ls{ACA_0}]
  Each $T\in C^*([0,1])$ is of the form \eqref{eq:cdinf} for a suitable (code of an) $f\in BV$.
\end{proposition}

We just note that since $h$ can be approximated by infinitely often differentiable functions we may assume that it is differentiable. Then one can use integration by parts on \eqref{eq:cdinf} and obtain that
\[
T(h) = \lim_{k \to \infty} \left( h(1)p_k(1) - h(0)p_k(0) - \int_0^1 h' \cdot p_k \, dx \right)
.\]
Under the assumption that $h(0)=h(1)=0$---this is given for instance if $h\in C_0((0,1))$, that is the space of all uniformly continuous functions with compact support included in $(0,1)$---we get
\[
T(h) = - \lim_{k\to \infty} \int_0^1 h' \cdot p_k \, dx
.\]
This value can be computed from ${\lVert h' \rVert}_\infty$ since ${\lVert p_k - p_{k+1} \rVert}_1 \le 2^{-k}$. 
Thus one obtains the following.
\begin{proposition}[\ls{RCA_0}]\label{pro:intpart}
  The functional $T(h)$ as in \eqref{eq:cdinf} restricted to $h\in C_0((0,1))$ exists and does only depend on the $L_1$-equivalence class of $f$ (and not on its code).
\end{proposition}
Or in other words, in this restricted case one does not need \ls{ACA_0} to get \prettyref{pro:bccduleq1}. The proposition below shows that \ls{ACA_0} is in general necessary.

\begin{proposition}[\ls{RCA_0}]\label{pro:bccduleq1eq}
  The statement of \prettyref{pro:bccduleq1} is equivalent to \ls{ACA_0}.

  In fact, it suffices to know for each $f\in BV$ the value $\lVert T \rVert$ or $V_{L_1}(f)$ for $T$ as in \eqref{eq:cdinf} to obtain \ls{ACA_0}.
\end{proposition}
\begin{proof}
  The right-to-left direction is \prettyref{pro:bccduleq1}. For the other direction consider the $\Pi^0_1$-statement (indexed by $n$)
  \[
  \Forall{i} \phi(n,i)
  .\]
  We show that we can build a set $X$ with $n\in X \IFF \Forall{i} \phi(n,i)$.

  Let
  \[
  f_{n,k}(x) := 
  \begin{cases}
    1- 2 \int_0^x \eta_{2^{-i'-1}}(y) \, dy & \text{if } \Exists{i\le k} \phi(n,i) \text{ and $i'$ is minimal with $\phi(n,i')$}, \\
    0 & \text{otherwise.}
  \end{cases}
  \]
  Since ${\left\lVert 1- 2\int_0^x \eta_{2^{-i'-1}}(y) \, dy\right\rVert}_1 < 2^{-i'-1}$ the sequence $(f_{n,k})_k$ forms a Cauchy-sequence with rate $2^{-k}$ for each $n$ and the variation is bounded by $1$. By \prettyref{lem:l1bv} the limit of $f_n$ of $(f_{n,k})_k$ is contained in $BV$. 

  Let $T_n$ be the functional corresponding to $f_n$ as in \eqref{eq:cdinf}. Since the function $f_n$ is the constant $0$ function if $\Forall{i} \phi(n,i)$ is true and otherwise $\lambda x . 1- 2 \int_0^x \eta_{2^{-i'-1}}(y) \, dy$ for an $i'$ we get that
  \begin{align*}
    T_n(\lambda x. 1) = 0  &\IFF \Forall{i} \phi(n,i) \\
    T_n(\lambda x. 1) = -1 &\IFF \NOT \Forall{i} \phi(n,i)
  \end{align*}
  Thus, one can read off the real number $T_n(\lambda x .1)$ whether $\Forall{i} \phi(n,i)$ is true. To obtain the second statement of the proposition for this particular $n$ note that since $T_n$ is non-increasing $\lVert T_n \rVert = V_{L_1}(f) = -T_n(\lambda x. 1)$.

  To obtain the full result we use a standard Cantor-middle third set construction to embed the Cantor space  into the unit interval. See for instance the proof of Theorem~IV.1.2 in \cite{sS09}.
  Thus, let $f(x):= \sum_{n=0}^\infty \frac{2 f_n(x)}{3^n}$ and let $T$ be the corresponding functional.
  Then $\Forall{i} \phi(0,i)$ if true if $-T(\lambda x. 1) \in [0,1/3]$ and false if it is in $[2/3,1]$. The statement for $n=1$ is true if $-T(\lambda x. 1) \in [0,1/9] \cup [2/3,7/9]$ and false if it is in $[2/9,1/3]\cup[8/9,1]$ and so on. From this one can easily construct the set $X$.
\end{proof}

\begin{remark}[Weak$^*$ topology]
  The space $C^*([0,1])$ is a dual space and, with this, one can define the weak$^*$ topology on it in the usual way. We say a sequence $(T_n)_n\subseteq C^*([0,1])$ converges to $T$ in the \emph{weak$^*$ topology} if{f}
  \[
  T_n(h) \xrightarrow{n\to \infty}  T(h) \quad\text{for all $h\in C([0,1])$}
  .
  \]
  Since $BV$ is isomorphic to $C^*([0,1])$ this induces a topology on $BV$. However, in most cases the following combination with the $L_1$\nobreakdash-topology is used.
  We say that a sequence of functions $(f_n)\subseteq BV$ converges in the \emph{weak$^*$ topology} to $f$ if{f} $f_n\xrightarrow{n\to\infty} f$ in $L_1$ and the functionals corresponding to $f_n$ converge in weak$^*$ topology of $C^*([0,1])$. See Definition~3.11 of \cite{AFP00}.

  One can show that for a sequence $(f_n)_n$ and $f$ in $BV$ that if 
  \begin{itemize}[label=$-$]
    \item $f_n \xrightarrow{n\to\infty} f$ in $L_1$ and 
    \item the variation of $(f_n)_n$ is uniformly bounded
  \end{itemize}
  then there exists a subsequence $(f_{g(n)})_n$ converging in the weak$^*$ topology to $f$.
  See Proposition~3.13 in \cite{AFP00}.\footnote{Note that the theorem there is stated in a misleading way. The statement should actually read ``\textbf{Proposition 3.13} Let $(u_h) \subset [BV(\omega)]^m$. Then there exists a subsequence $(u_{k(h)})$ converging to $u$ in $[BV(\omega)]^m$ if $(u_h)$ is bounded in $[BV(\omega)]^m$ and $u_h$ converges to $u$ in $[L_1(\omega)]^m$. If $(u_{h})$ converges to $u$ in $[BV(\omega)]^m$ then $u_h$ converges to $u$ in $[L_1(\omega)]^m$ and is bounded in $[BV(\omega)]^m$.''}

  This leads to the following.
  The representation of $BV$ as given in \prettyref{def:bv} is \emph{consistent} with the weak$^*$ topology in the sense that if a sequence of representations $(r_i)_i\subseteq \Nat^\Nat$ converges in the Baire space then the sequence of represented elements $f_{r_i}$ contains a weak$^*$-\hspace{0pt}converging subsequence. See also \prettyref{lem:l1bv}.
\end{remark}

\subsection{Other representations}

In \cite{vB05} Brattka proposes two different ways to represent elements of non-separable spaces. The first representation essentially codes an element $f$ of a space $X$ as a sequence of countable objects plus the norm ${\lVert f \rVert}_X$. Whereas the second representation just consists of the countable objects plus an upper bound $v$ on the norm. See also \cite{BS05}.

In the case of \prettyref{def:bv} the countable objects are rational polynomials.
The representation we defined in \prettyref{def:bv} is intermediate between those two representations proposed by Brattka because we have an upper bound of the norm of an element $f\in BV$, i.e.~${\Vert f \rVert}_{BV} \le v$, and thus the second representation is reducible to our representation. However, we have $f$ as full $L_1$ object including its norm, thus our representation is stronger.

Alternatively, we could have added the value of the variation instead of merely an upper bound to the representation of an element of $BV$. Since by \prettyref{pro:bccduleq1eq} going from an upper bound to right value of $V_{L_1}$ requires \ls{ACA_0}, this representation is too strong in general.

Other ways to represent functions of bounded variation are to take computable functions with a computable variation, see \cite{RZB02}, or as a computable function defined on a countable, dense subset of $[0,1]$, see \cite{HW07,JPW13}.
The first approach is too restricted since very few functions of bounded variation are computable. The second approach is orthogonal to ours since it defines points of functions, whereas we define the function in the $L_1$\nobreakdash-sense. This representation has been successfully used in algorithmic randomness, see \cite{BMNta,jRta}. However we believe that our approach is more natural since it fits nicely into the Sobolev spaces and easily generalizes to functions defined in $\Real^n$, which is not the case for the pointwise definition.

\section{Helly's selection theorem}\label{sec:hst}

\begin{theorem}[Helly's selection theorem, \lp{HST}, \ls{ACA_0}]\label{thm:helly}
  Let $(f_n)_n\subseteq BV$ be a sequence of functions with bounds for variations $v_n$. 
  If
  \begin{enumerate}[label=(\roman*)]
  \item\label{enum:hst1} ${\left\lVert f_n \right\rVert}_1 \le u$ for a $u\in \Rat$,
  \item\label{enum:hst2} $v_n \le v$ for a $v\in \Rat$,
  \end{enumerate}
  then there exists an $f\in BV$ and a subsequence $f_{g(n)}$ such that $f_{g(n)} \xrightarrow{n\to\infty} f$ in $L_1$  and the variation of $f$ is bounded by $v$.
\end{theorem}
The statement of this theorem will be abbreviated by \lp{HST}.

Originally Helly's selection theorem was formulated for usual functions and not $L_1$\nobreakdash-\hspace{0pt}functions. There usually \ref{enum:hst1} is replaced by the statement that $\abs{f_n(x)} \le u'$ for an $x\in[0,1]$ and a bound $u'$. Note that this implies \ref{enum:hst1} since by \ref{enum:hst2} with the bound $u'$ we have ${\lVert f_n \rVert}_\infty \le u'+v$ and with this also ${\lVert f_n \rVert}_1 \le u'+v =: u$.

For the proof of \lp{HST} we will need the following lemma.
\begin{lemma}[\ls{RCA_0}]
  Let $f\in BV$ and let $v$ be the  bound of variation of $f$. The system $\ls{RCA_0}$ proves that for each $\epsilon > 0$ that
  \begin{enumerate}[label=(\roman*)]
  \item\label{enum:moll2:1} $f^\epsilon \in L_1 $ exists, and that
  \item\label{enum:moll2:2} ${\lVert f^\epsilon - f \rVert}_1 \le 2 \epsilon  v$.
  \end{enumerate}
\end{lemma}
\proof
  Let $(p_k)_k$ be the sequence of rational polynomials coding $f$. We have 
  \begin{align*}
    {\left\lVert f^\epsilon - {(p_k)}^\epsilon \right\rVert}_1 & 
    = \int_0^1 \int_{-\epsilon}^{\epsilon} \eta_\epsilon(y) \left(f(x-y) - p_k(x-y)\right) \, dy \, dx \\
    & = \int_{-\epsilon}^{\epsilon} \eta_\epsilon(y) \int_0^1  \left(f(x-y) - p_k(x-y)\right) dx \, dy \qquad\text{by Fubini}\\
    & \le 2 {\lVert f - p_k \rVert}_1 \int _{-\epsilon}^{\epsilon} \eta_\epsilon(y) = 2 {\lVert f - p_k \rVert}_1 .
  \end{align*}
  (The $2$ in the above inequality comes from the possible reflection of $f$ in the mollification as we defined it.)
  It follows that (a $2^{-k+1}$-good approximation with rational polynomials of) $(p_{k+2})^\epsilon$ is a code for $f^\epsilon\in L_1$. 

  For \ref{enum:moll2:2} we have for any $k$
  \begin{alignat*}{3}
    {\lVert f^\epsilon - f \rVert}_1 & \le {\lVert (p_{k})^\epsilon - p_k \rVert}_1 + 2^{-k+2} 
    \shortintertext{since ${\lVert f-p_k\rVert}_1 < 2^{-k}$ and ${\lVert f^\epsilon-p_k^\epsilon\rVert}_1 < 2^{-k+1}$ by the above estimate. Further, }
    {\lVert (p_{k})^\epsilon - p_k \rVert}_1 
    & = \int_0^1 \int_{-\epsilon}^{\epsilon} \eta_\epsilon(y) \cdot p_k(x-y) \, dy - p_k(x) \, dx \\
    & = \int_0^1 \int_{-\epsilon}^{\epsilon} \eta_\epsilon(y) \cdot \left(p_k(x-y) - p_k(x)\right) \, dy \, dx & \text{since $\int \eta_\epsilon = 1$}\\
    & = \int_0^1 \int_{-1}^{1} \eta(y) \cdot \left(p_k(x-\epsilon y) - p_k(x)\right) \, dy \, dx & \text{substituting $y\mapsto \epsilon y$}
    \shortintertext{since $\abs{p_k(x-\epsilon y) - p_k(x)} = \abs{\int_{0}^{y} \frac{d}{dy} p_k(x-\epsilon y)\, dy} = \abs{\epsilon \int_{0}^{y}  p_k'(x-\epsilon y) \,dy} \le 2 \epsilon {\lVert p_k' \rVert}_1$ for $y\in [-1,1]$}
    & \le \int_0^1 \int_{-1}^{1} \eta(y) \cdot 2 \epsilon {\lVert p_k'
      \rVert}_1 \, dy \, dx = 2 \epsilon {\lVert p_k' \rVert}_1
    \;\smash[b]{\le 2 \epsilon v.}\rlap{\hbox to 119 pt{\hfill\qEd}}\medskip
    \end{alignat*}

\begin{proof}[Proof of \prettyref{thm:helly}]
  For the mollifications $f_n^\epsilon$ of $f_n$ we have by definition \eqref{eq:defmollification} that
  \begin{alignat*}{4}
  {\left\lVert f_n^\epsilon \right\rVert}_\infty & \le  {\lVert f_n \rVert}_1\, \lVert {\eta_\epsilon \rVert}_\infty &&\le \frac{u}{\epsilon} ,
  \shortintertext{and by \eqref{eq:moldif2} that}
  {\left\lVert \left(f_n^\epsilon\right)' \right\rVert}_\infty & \le {\lVert f_n \rVert}_1\, \lVert {\eta_\epsilon' \rVert}_\infty && \le u \, \lVert {\eta_\epsilon' \rVert}_\infty .
  \end{alignat*}
  Thus, for each fixed $\epsilon$ the sequence ${(f_n^\epsilon)}_n$ is uniformly bounded and---by the uniform bound on the derivative---equicontinuous.
  We instantiate $\epsilon$ with $2^{-i}$ and obtain a sequence of sequences of bounded, equicontinuous functions ${(f_n^{(2^{-i})})}_{n,i}$. By the previous lemma this sequence is contained in $L_1$ and converges as $i\to \infty$ to $f_n$.

  By \prettyref{pro:aadiag} below, a variant of the Arzelà-Ascoli theorem, there exists a subsequence $g(n)$, such that for each $k$
  \[
  \Forall{j\le k} \Forall{n,n'\ge k} {\left\lVert f_{g(n)}^{(2^{-j})} - f_{g(n')}^{(2^{-j})}\right\lVert}_\infty \le 2^{-k} .
  \]
  Now for $n,n'\ge k$
  \begin{align*}
    {\left\lVert f_{g(n)} - f_{g(n')} \right\rVert}_1 & \le {\left\lVert f_{g(n)}^{(2^{-k})} - f_{g(n')}^{(2^{-k})} \right\rVert}_1 + {\left\lVert f_{g(n)} - f_{g(n)}^{(2^{-k})} \right\rVert}_1 + {\left\lVert f_{g(n')} - f_{g(n')}^{(2^{-k})} \right\rVert}_1 \\ 
    & \le 2^{-k} + 2\cdot 2 \cdot 2^{-k} v
    .
  \end{align*}
  Thus, $f_{g(n)}$ forms a $L_1$\nobreakdash-converging sequence with rate of convergence $2^{-k} + 2^{-k+2} v$. Thus $\lim f_{g(n)} = f \in L_1$. By \prettyref{lem:l1bv} we have that $f\in BV$. 
\end{proof}
The previous proof was inspired by \cite[Theorem~3.23]{AFP00}.

\begin{proposition}[Diagonalized Arzelà-Ascoli, \ls{ACA_0}]\label{pro:aadiag}
  Let $f_{n,j}\colon [0,1] \longto \Real$ be a sequence of sequences of functions. If for each $j$
  \begin{enumerate}
  \item the sequence $(f_{n,j})_n$ is bounded by $u_j\in \Rat$, and
  \item $(f_{n,j})_n$ is uniformly equicontinuous, i.e., there exists a modulus of uniform equicontinuity $\phi_j(l)$, such that 
    $\Forall{l}\Forall{n}\Forall{x,y\in[0,1]} \big(\abs{x-y} < 2^{-\phi_j(l)}\linebreak[1] \IMPL \abs{f_{n,j}(x)-f_{n,j}(y)} < 2^{-l}\big)$,
  \end{enumerate}
  then there exists a subsequence $g(n)$ such that for all $j$ the sequence $f_{g(n),j}$ converges uniformly in the sense that
  \begin{equation}\label{eq:conv}
  \Forall{k} \Forall{j\le k} \Forall{n,n'\ge k} {\left\lVert f_{g(n),j} - f_{g(n'),j} \right\lVert}_\infty < 2^{-k}
  .\end{equation}
\end{proposition}
\begin{proof}
  By replacing $f_{n,k}$ with $\frac{ f_{n,k}}{2 u_n} + \frac{1}{2}$ we may assume that the image of $f_{n,k}$ is contained in the unit interval $[0,1]$.
 
  In \cite[Lemma~3, Corollary~4]{aK14a} we showed that an equicontinuous sequence of functions $h_n\colon [0,1]\longto[0,1]$ converges uniformly if{f} $h_n$  converges pointwise on $\Rat \cap [0,1]$, i.e., for an enumeration $q$ of $\Rat \cap [0,1]$ the sequence ${\big({(h_n(q(i)))}_i\big)}_n \subseteq [0,1]^\Nat$ converges in $[0,1]^\Nat$  with the product norm $d((x_i),(y_i)) = \sum_i 2^{-i} d(x_i,y_i)$.
  Moreover, from a rate of convergence and the modulus of uniform equicontinuity one can calculate a rate of convergence of $h_n$ in ${\lVert \cdot \rVert}_\infty$. 

  With this the Arzelà-Ascoli theorem follows directly from an application of the Bolzano-Weierstra{\ss} principle for the space $[0,1]^\Nat$. For details see \cite{aK14a}.

  We can parallelize this process for $f_{n,j}$ by applying the Bolzano-Weierstra{\ss} principle to the sequence 
  ${\big(\big({f_{n,j}(q(i))}\big)_{\langle i,j \rangle}\big)}_n \subseteq [0,1]^\Nat$. With this we obtain a subsequence $g(n)$ such that for each $j$ we have  $\big({f_{g(n),j}(q(i))}\big)_i\in{[0,1]}^\Nat$ converges at a given rate for $n\to \infty$.  By the above considerations we get that $f_{g(n),j}\in C([0,1])$ converges uniformly at a given rate (depending in $\phi_j$). By thinning out the sequence $g(n)$ we get \eqref{eq:conv}.

  This proposition is provable in \ls{ACA_0} since Bolzano-Weierstra{\ss} principle for the space $[0,1]^\Nat$ is instance-wise equivalent to the Bolzano-Weierstra{\ss} principle for $[0,1]$ which is provable in \ls{ACA_0}, see e.g.~\cite{aK14a,sS09}.
\end{proof}

We now come the reversal.
\begin{theorem}\label{thm:hstaca}
  Over \ls{RCA_0}, \lp{HST} is equivalent to \lp{ACA_0}.
\end{theorem}
\begin{proof}
  The right to left direction is simply \prettyref{thm:helly}. For the left-to-right direction we will show that \lp{HST} implies the Bolzano-Weierstra{\ss} principle (for $[0,1]$) which is by \cite[Theorem~III.2.2]{sS09} equivalent to \lp{ACA_0}. Let $(x_n)_n\subseteq [0,1]$ be any sequence in the unit interval. Let $f_n(x) := x_n$ be the sequence of corresponding constant functions. It is clear that $f_n\in BV$ and that ${\lVert f_n \rVert}_1 = x_n$. One easily verifies that for any limit $f$ as given by \lp{HST} the value ${\lVert f \rVert}_1$ is a limit point of $x_n$ and thus a solution to \lp{BW}. 
\end{proof}

The proofs of \prettyref{thm:helly} and \prettyref{thm:hstaca} actually give more information on the strength of \lp{HST}. It shows that for each instance of \lp{HST}, that is for each sequence of functions $(f_n)_n\subseteq BV$ with a uniform bound of variation, one can compute  uniformly a sequence $(x_n)_n\subseteq [0,1]$, such that from any limit point of this sequence one can compute a solution to \lp{HST} for $f_n$. By the proof of \prettyref{thm:hstaca} the backward direction also holds.
This  is summarized in the following corollary.
\begin{corollary}
  The principles \lp{HST} and \lp{BW} are instance-wise equivalent, i.e., writing \lpp{HST}{(f_n)} for \lp{HST} restricted to $(f_n)$ and \lpp{BW}{(x_n)}  for \lp{BW} restricted to $(x_n)$, then we have the following.
  There are codes for Turing machines $e_1,e_2$, such that
  \begin{align*}
    \ls{RCA_0} &\vdash \Forall{X} \left(\lpp{BW}{\{e_1\}^X} \IMPL \lpp{HST}{X}\right), \\
    \ls{RCA_0} &\vdash \Forall{X} \left(\lpp{HST}{\{e_2\}^X} \IMPL \lpp{BW}{X}\right). 
  \end{align*}
\end{corollary}
This corollary should be compared with Theorem~3.1 of \cite{aK11}, where it is shown that \lp{BW} is instance-wise equivalent to \lp{WKL} for $\Sigma^0_1$\nobreakdash-trees, and Theorem~9 of \cite{aK14a}, where it is shown that \lp{BW} is instance-wise equivalent to the Arzelà-Ascoli theorem.

\begin{remark}[\lp{HST_{weak}}]
  In \cite{aK11} we also analyzed the following weaker variant \lp{BW_{weak}} of the Bolzano-Weierstrass principle, which states that for each sequence $(x_n)\subseteq [0,1]$ there is a subsequence that converges but possibly without any computable rate of convergence. Since points are coded as sequences converging at the rate $2^{-k}$ the existence of the limit point of the sequence might not be provable. This principle is considerably weaker than \lp{BW}. For instance \lp{BW_{weak}} it does not imply \lp{ACA_0} nor \lp{WKL_0}.

  Replacing \lp{BW} in the above proof immediately yields that \lp{BW_{weak}} is instance-wise equivalent to the variant of \lp{HST} which only states the existence of a converging subsequence. 
\end{remark}

\subsection{\lp{HST} in the Weihrauch lattice}

Helly's selection theorem can be formulated in the Weihrauch lattice. The above proof yields also a classification in these terms.
We refer the reader to \cite{BG11,BGM12} for an introduction to the Weihrauch lattice.

The functions of the space $L_1$ can be represented by the rational polynomials closed under the $\lVert \cdot \rVert_1$\nobreakdash-norm. We will call this representation $\delta_{L_1}$. With this Helly's selection theorem is then a partial multifunction of the following type.
\[
\mathsf{HST} :\subseteq \left(L_1([0,1]),\delta_{L_1}\right)^\Nat \rightrightarrows \left(L_1([0,1]),\delta_{L_1}\right)
\]
where $\textrm{dom}(\mathsf{HST}) = \left\{\, (f_n) \sizeMid \int_0^1 \abs{f_n'}\, dx \text{ is uniformly bounded}\, \right\}$. The derivative of $f_n'$ here is taken in the sense of distributions. We chose this representation since it is customary in the Weihrauch lattice not to add any additional information---like the uniform bound on the variation---to the representation. However, we can easily recover this uniform bound by searching for it. This can be done using the limit $\mathrm{lim}$. Since $v$ is not needed to build the sequence of equicontinuous functions ($f_n^{(2^{-i})}$ in the proof of \prettyref{thm:helly}) the bound of the variation can be computed in parallel to the application of the diagonalized Arzelà-Ascoli theorem (which follows from $\mathsf{BWT_{[0,1]^\Nat}}$) . Thus,we get
\[
\mathsf{HST} \le_{\mathrm{{W}}} \mathrm{lim} \times \mathsf{BWT_{[0,1]^\Nat}} \le_{\mathrm{W}} \mathsf{BWT_{[0,1]^\Nat}} \equiv_{\mathrm{W}} \mathsf{BWT}_{\Real}
.\]
Using the reversal we obtain in total $\mathsf{HST} \equiv_{\mathrm{W}} \mathsf{BWT}_{\Real}$.

\section*{Acknowledgments}
I thank Jason Rute for useful discussions, and the two anonymous referees for remarks that lead to an improved presentation, to \prettyref{pro:just1con} (in particular its proof in \lp{WWKL_0}), and \prettyref{pro:bccduleq1eq}.

\bibliographystyle{acm}
\bibliography{bv.final2}

\end{document}